 \documentclass[10pt,twocolumn,twoside]{IEEEtran}

\usepackage{cite}      
\usepackage{graphicx}  
\usepackage{psfrag}    
\usepackage{subfigure} 
\usepackage{url}       
\usepackage{amssymb,amsmath,amscd,amsfonts,amstext,bbm, enumerate, ntheorem,eclbkbox}
\usepackage{array}
\usepackage{multirow}
\usepackage[latin1]{inputenc}
\usepackage{hyperref}

\usepackage[textwidth=1.6cm, textsize=footnotesize]{todonotes}  
\input{numerical_header.tex}

\definecolor{link}{RGB}{240,90,90}
\definecolor{activelink}{RGB}{60,179,113}
\definecolor{inactivelink}{RGB}{75,75,75}

\definecolor{extblue}{RGB}{0,0,255}
\definecolor{intblue}{RGB}{0,140,195}

\definecolor{extgray}{RGB}{100,100,100}
\definecolor{intgray}{RGB}{175,175,175}

\xdefinecolor{rouge}{rgb}{0.6,0.1,0}
\definecolor{trugreen}{rgb}{0.0, 0.33, 0.0}

\newcommand{\leftnorm}{\left|\!\left|\!\left|}
\newcommand{\rightnorm}{\right|\!\right|\!\right|}

\DeclareMathOperator*{\proj}{proj}

\newcommand\oT{\mathsf{T}}

\newcommand\oU{\mathsf{U}}

\newcommand\oV{\mathsf{V}}
\newcommand\oI{\mathsf{I}}

\newcommand{\prox}{\mathop{\mathrm{prox}}\nolimits}

\theoremstyle{plain} 
\newtheorem{lemma}{Lemma}

\newtheorem{theorem}{Theorem}

\theoremstyle{definition}
\newtheorem{assumption}{Assumption}

\theoremstyle{remark} 
\newtheorem{remark}{Remark}

\newcommand{\cH}{{\cal H}}

\newcommand{\NN}{{\mathbb{N}}}

\newcommand{\un}{{\bs 1}}
\DeclareMathOperator*{\fix}{fix}

\DeclareMathOperator*{\ri}{ri}

\DeclareMathOperator*{\dom}{dom}
\DeclareMathOperator*{\argmin}{argmin}
\newcommand{\bs}{\boldsymbol}
\newcommand{\EE}{{\mathbb{E}}}
\def\adots{
  \mathinner{\mkern1mu\raise1pt\hbox{.}\mkern2mu\raise4pt\hbox{.}
  \mkern2mu\raise7pt\vbox{\kern7pt\hbox{.}}\mkern1mu}}
\def\build#1_#2^#3{\mathrel{
\mathop{\kern 0pt#1}\limits_{#2}^{#3}}}
\IEEEoverridecommandlockouts

\begin{document}

\title{A Coordinate Descent Primal-Dual Algorithm \\ 
and Application to Distributed {Asynchronous} Optimization}

\author{P. Bianchi, W. Hachem and F. Iutzeler
 \thanks{The first two authors are with the CNRS LTCI; Telecom ParisTech,  
  Paris, France. The third author is with LJK, Université Joseph Fourier, Grenoble, France.  
  E-mails: \texttt{pascal.bianchi, walid.hachem@telecom-paristech.fr,
  franck.iutzeler@imag.fr}. 
  This work was granted by the French Defense Agency (DGA) ANR-grant
  ODISSEE.}}  



\maketitle

\begin{abstract}  Based on the idea of randomized coordinate descent of
$\alpha$-averaged operators, a randomized primal-dual optimization algorithm is
introduced, where a random subset of coordinates is updated at each iteration.
The algorithm builds upon a variant of a recent (deterministic) algorithm
proposed by V\~u and Condat that includes the well known ADMM as a particular
case. The obtained algorithm is used to solve asynchronously a distributed
optimization problem. A network of agents, each having a separate cost function
 containing a differentiable term, seek to find a consensus on the
minimum of the aggregate objective. The method yields an algorithm where at
each iteration, a random subset of agents wake up, update their local
estimates, exchange some data with their neighbors, and go idle. Numerical
results demonstrate the attractive performance of the method. \\ 
The general approach can be naturally adapted to other situations where
coordinate descent convex optimization algorithms are used with a random choice
of the coordinates. 

\end{abstract}


\section{Introduction}

Let ${\cal X}$ and ${\cal Y}$ be two Euclidean spaces and let $M :  
{\cal X} \to {\cal Y}$ be a linear operator. 
Given two real convex functions $f$ and $g$ on $\cal X$ and a real convex
function $h$ on $\cal Y$, we consider the minimization problem  
\begin{equation}
\label{pb} 
\inf_{x\in\cal X}  f(x) + g(x) + h(Mx) 
\end{equation}
where  $f$ is differentiable and its gradient $\nabla f$ is
Lipschitz-continuous. {\color{black} Although our theoretical contributions
are valid for very general functions $f$, $g$ and $h$, the application part of this paper
puts a special emphasis on the problem of distributed optimization. In this particular framework,
one considers a set of $N$ agents such that each agent $n=1,\dots,N$ has a private cost 
of the form $f_n+g_n$ where $f_n$ and $g_n$ are two convex cost function on some (other) space $\overline {{\cal X}}$, 
$f_n$ being differentiable. The aim is to distributively solve
\begin{equation}
\inf_{u\in \overline{\cal X}} \sum_{n=1}^N f_n(u) + g_n(u)\ .\label{eq:1}
\end{equation}
In order to construct distributed algorithms a standard approach consists in
introducing
$$
f(x) = \sum_{n=1}^N f_n(x_n)\quad \text{and}\quad g(x) = \sum_{n=1}^N g_n(x_n)
$$
for all $x=(x_1,\dots,x_N)$ in the product space ${\cal X}=\overline{\cal X}^N$.
Obviously, problem~(\ref{eq:1}) is equivalent
to the minimization of $f(x)+g(x)$ under the constraint that all components of $x$ are equal
\emph{i.e.}, $x_1=\dots =x_N$. Therefore, Problem~(\ref{eq:1}) is in fact a special instance
of Problem~(\ref{pb}) if one chooses $h(Mx)$ as an indicator function, equal to zero if $x_1=\dots =x_N$ and to $+\infty$ otherwise. As we shall see, this reformulation of Problem~(\ref{eq:1}) is often 
a mandatory step in the construction of distributed algorithms.}

Our contributions are as follows. 
\begin{enumerate}
\item V\~u and Condat have separately proposed an algorithm to solve (\ref{pb})
in \cite{vu2013splitting} and \cite{cond-13} respectively.
Elaborating on this algorithm, we provide an iterative algorithm for
solving~(\ref{pb}) which we refer to as ADMM+ (Alternating
Direction Method of Multipliers plus) because it includes the well
known ADMM~\cite{gabay1976dual,gab-83} as {the special case corresponding to 
$f=0$. Interestingly, in the framework of the distributed optimization, ADMM+ 
is provably convergent under weaker assumptions on the step sizes as compared 
to the original V\~u/Condat algorithm.}  

\item {Based on the idea of the \emph{stochastic coordinate descent} who has
been mainly studied in the literature in the special case of proximal gradient
algorithms \cite{nesterov2012,fercoq2013,bavcak2013}, we develop a
\emph{distributed asynchronous} version of ADMM+.  As a first step, we borrow
from \cite{vu2013splitting}-\cite{cond-13} the idea that their algorithm is an
instance of a so-called Krasnosel'skii-Mann iteration applied to an
$\alpha$-averaged operator~\cite[Section 5.2]{livre-combettes}. Such operators
have contraction-like properties that make the Krasnosel'skii-Mann iterations
converge to a fixed point of the operator. The principle of the stochastic
coordinate descent algorithms is to update only a random subset of coordinates
at each iteration.  In this paper, we show in most generality that a randomized
coordinate descent version of the Krasnosel'skii-Mann iterations still
converges to a fixed point of an $\alpha$-averaged operator. This provides 
as a side result a convergence proof of the stochastic coordinate descent 
versions of the proximal gradient algorithm, since this algorithm can be 
seen as the application of a $1/2$-averaged operator~\cite{livre-combettes}. 
More importantly in the context of this paper,
this idea leads to provably convergent asynchronous distributed versions of
ADMM+.}

\item 
{Putting together both ingredients above,
we apply our findings to 
asynchronous distributed optimization. First, the optimization
problem~\eqref{pb} is rewritten in a form where the operator $M$ encodes the
connections between the agents within a graph in a manner similar
to~\cite{sch-rib-gia-sp08}. Then, a distributed optimization algorithm for
solving Problem~(\ref{eq:1}) is obtained by applying ADMM+. 
Using the idea of
coordinate descent on the top of the algorithm, we then obtain a fully
asynchronous distributed optimization algorithm that we refer to as Distributed
Asynchronous Primal Dual algorithm (DAPD). At each iteration, an independent an
identically distributed random subset of agents wake up, apply essentially
the proximity operator on their local functions, send some estimates to
their neighbors and go idle.}  
\end{enumerate}


{An algorithm that has some formal resemblance with ADMM+ was proposed
in~\cite{ouyang2013}, who considers the minimization of the sum of two
functions, one of them being subjected to noise. This reference includes a
linearization of the noisy function in ADMM iterations.} 

{The use of stochastic coordinate descent on averaged operators has been
introduced in~\cite{iut-cdc13} (see also the recent preprint
\cite{pesquet-stochastic} which uses the same line of thought).  Note that the
approach of~\cite{iut-cdc13} was limited to unrelaxed firmly non expansive (or
$1/2$-averaged) operators, well-suited for studying ADMM which was the
algorithm of interest in~\cite{iut-cdc13}.}   

Asynchronous distributed optimization is a promising framework in order to scale up machine learning problems involving massive data sets
(we refer to \cite{boyd2011distributed} or the recent survey \cite{cevher2014convex}). 
Early works on distributed optimization include \cite{tsitsiklis:phd-1984,tsitsiklis:bertsekas:athans:tac-1986} where a network
of processors seeks to optimize some objective function known by all
agents (possibly up to some additive noise). More recently, numerous works
extended this kind of algorithm to more involved multi-agent scenarios,
see~\cite{kushner-siam87,lop-sayed-asap06,nedic:ozdaglar:parrilo:tac-2010,bia-for-hac-IT13,bia-jak-TAC13,nedic2013distributed,iutzeler2013explicit,tsianos2014efficient,mokhtariapproximate,jaggi2014communication,morral2014success,shi2014extra}.


Note that standard first order distributed optimization methods often rely on the so-called adaptation-diffusion approaches or variants.
The agents update their local estimates by evaluating their private gradient and then merge their estimate with their neighbors
using a local averaging step. 
Unfortunately, such methods require the use of a vanishing step size, which results in slow convergence. 
This paper proposes a first-order distributed optimization method with constant step size, which turns out to 
outperform standard distributed gradient methods, as shown in the simulations.

{\color{black} To the best of our knowledge, our method
is the first distributed algorithm combining the following attractive features:
  \begin{enumerate}
  \item The algorithm is asynchronous at the \emph{node-level}. 
Only a single node is likely to be active at a given iteration, only broadcasting the result of its computation
without expecting any feedback from other nodes.
This is in contrast with the asynchronous ADMM studied by \cite{iut-cdc13} and \cite{wei-ozd-arxiv13} which is
only asynchronous at the \emph{edge-level}. In these works, at least two connected nodes are supposed 
to be active at a common time. 
\item The algorithm is a \emph{proximal method}. Similarly to the distributed ADMM, it allows for the use of a proximity operator
at each node. This is especially important to cope with the presence of possibly non-differentiable regularization terms.
This is unlike the classical adaptation-diffusion methods mentioned above or the more recent first order distributed algorithm EXTRA proposed by\cite{shi2014extra}.
\item The algorithm is a \emph{first-order method}. Similarly to adaptation-diffusion methods, our algorithm allows
to compute gradients of the local cost functions. This is unlike the distributed ADMM which only admits implicit steps \emph{i.e.}, agents
are required to locally solve an optimization problem at each iteration.
\item The algorithm admits \emph{constant step size}. As remarked in \cite{shi2014extra}, standard adaptation-diffusion methods 
require the use of a vanishing step size to ensure the convergence to the sought minimizer. In practice, this comes at the price of slow convergence.
Our method allows for the use of a constant step size in the gradient descent step.
  \end{enumerate}
}

The paper is organized as
follows. Section~\ref{sec:pda} is devoted to the the introduction of
ADMM+ algorithm and its relation with the Primal-Dual algorithms
of V\~u\cite{vu2013splitting} and Condat~\cite{cond-13}, we also show
how ADMM+ includes both the standard ADMM and the Forward-Backward
algorithm (also refered to as proximal gradient algorithm) as special cases
{\color{black}\cite[Section 25.3]{livre-combettes}}. 
In Section~\ref{sec:cd}, we provide our result on
the convergence of Krasnosel'skii-Mann iterations with randomized
coordinate descent. 
Section~\ref{sec:distop} addresses the problem of asynchronous 
distributed optimization.
Finally, Section~\ref{sec:num} provides numerical results.

\section{A Primal Dual Algorithm} 
\label{sec:pda}

\subsection{Problem statement}

We consider Problem~(\ref{pb}).
Denoting by $\Gamma_0({\cal X})$ the set of proper lower semi-continuous convex 
functions on ${\cal X} \to (-\infty, \infty]$ and by $\| \cdot \|$ the norm 
on ${\cal X}$, we make the following assumptions: 
\begin{assumption}
\label{cvx}
The following facts hold true: 
\begin{enumerate}[(i)]
 \item $f$ is a convex differentiable function on ${\cal X}$,
 \item $g\in \Gamma_0({\cal X})$ and $ h\in \Gamma_0({\cal Y})$.
  \end{enumerate}
\end{assumption}


We consider the case where $M$ is injective (in particular, it is implicit 
that $\text{dim}({\cal X})\leq \text{dim}({\cal Y})$).
In the latter case, we denote by ${\cal S}=\text{Im}(M)$ the image of $M$ and 
by $M^{-1}$ the inverse of $M$ on ${\cal S}\to{\cal X}$. We emphasize the fact 
that the inclusion $\cal S \subset \cal Y$ might be strict. 
We denote by $\nabla$ the gradient operator.
\begin{assumption}
\label{hyp:M2} The following facts hold true:
 \begin{enumerate}[(i)]
 \item $M$ is injective\,,
 \item $\nabla (f\circ M^{-1})$ is $L$-Lipschitz continuous on $\cal S$.
 \end{enumerate}
\end{assumption}

We denote by $\dom q$ the domain of a function $q$ and by $\ri S$ the 
relative interior of a set $S$ in a Euclidean space. 
\begin{assumption}
\label{qualif}
The infimum of Problem~\eqref{pb} is attained. Moreover, the following 
qualification condition holds 
\[
0 \in \ri(\dom h - M \dom g)
\] 
{where $M \dom g$ is the image by $M$ of $\dom g$}. 
\end{assumption}

The \emph{dual} problem corresponding to the \emph{primal} problem~\eqref{pb} 
is written 
\[
\inf_{\lambda \in {\cal Y}}  (f + g)^*(-M^*\lambda) + h^*(\lambda) 
\]
where $q^*$ denotes the Legendre-Fenchel transform of a function $q$ and
where $M^*$ is the adjoint of $M$. With the
assumptions~\ref{cvx} and~\ref{qualif}, the classical Fenchel-Rockafellar 
duality theory \cite{Roc70,livre-combettes} shows that 
\begin{multline} 
\min_{x\in\cal X}  f(x) + g(x) + h(Mx)  \\
= - 
\inf_{\lambda \in {\cal Y}}  (f + g)^*(-M^*\lambda) + h^*(\lambda),
\label{duality} 
\end{multline} 
and the infimum at the right hand member is attained. Furthermore, denoting by
$\partial q$ the subdifferential of a function $q \in \Gamma_0({\cal X})$,
any point $(\bar x, \bar\lambda) \in {\cal X} \times {\cal Y}$ at which the above 
equality holds satisfies 
\[
\left\{
  \begin{array}[h]{l}
    0\in \nabla f(\bar x) + \partial g(\bar x) + M^*\bar\lambda \\
    0 \in -M \bar x + \partial h^*(\bar\lambda)  
  \end{array}
\right.
\] 
and conversely. Such a point is called a \emph{primal-dual} point. 

\subsection{A Primal-Dual Algorithm}
\label{sec:algo}

We denote by $\langle\,\cdot,\cdot\,\rangle$ the inner product on
$\cal X$.  We keep the same notation $\|\cdot\|$ to represent the norm
on both $\cal X$ and $\cal Y$.  For some parameters $\rho,\tau>0$, we
consider the following algorithm which we shall refer to as
ADMM+.\smallskip

\noindent   \textbf{ADMM+} 
\vspace*{-.1cm} 
  \begin{subequations} 
    \begin{align}
z^{k+1} &= \argmin_{z\in {\cal Y}}
\Bigl[ h(z)+\frac{\|z-(Mx^{k}+\rho\lambda^k)\|^2}{2\rho} \Bigr] \label{adm+-z}\\
\lambda^{k+1} &= \lambda^k + \rho^{-1}(Mx^{k}-z^{k+1}) \label{adm+-l} \\
      u^{k+1} &= (1-\tau\rho^{-1})Mx^{k} + \tau\rho^{-1}z^{k+1} 
\label{adm+-u} \\ 
x^{k+1} &= \argmin_{x\in {\cal X}} \Bigl[ g(x)+\langle\nabla f(x^k),x\rangle 
           \nonumber \\
    & 
  \ \ \ \ \ \ \ \ \ \ \ \ \ \ \ \ 
     + \frac {\|Mx-u^{k+1}+\tau\lambda^{k+1}\|^2}{2\tau} \Bigr] \label{adm+-x} 
    \end{align}
  \end{subequations} 

{
This algorithm is especially useful in the situations where $\nabla f$ and 
the left hand member of Equation~\eqref{adm+-x} are both easy to compute,
as it is the case when (say) $f$ is quadratic and $g$ is an $\ell_1$ 
regularization term. In such situations, working directly on $f+g$ is often 
computationally demanding.} 

\begin{theorem}
\label{the:cv}
Let Assumptions~\ref{cvx}--\ref{qualif} hold true. Assume that  
$\tau^{-1} - \rho^{-1} > L/2$. 
For any initial value $(x^0, \lambda^0) \in {\cal X} \times {\cal Y}$, the 
sequence $(x^k, \lambda^k)$ defined by ADMM+ converges to a primal-dual point 
$(x^\star,\lambda^\star)$ of~\eqref{duality} as $k\to\infty$. 
\end{theorem}

\begin{remark}
  In the special case when $f=0$ (that is $L=0$), it turns out that 
the condition $\tau^{-1} - \rho^{-1} > L/2$ can be further weakened
to $\tau^{-1} - \rho^{-1} \geq 0$ (see~\cite{cond-13}). 
It is therefore possible to set $\tau=\rho$ and thus have a single instrumental parameter to tune in the algorithm.
Note also that in $f=0$ the algorithm is provably convergent with no need to require the injectivity of $M$.
\end{remark}

The proof of Theorem~\ref{the:cv} is provided in Appendix~\ref{sec:proofCV}. It is based on Theorem~\ref{vu} below.
For any function $g\in \Gamma_0({\cal X})$ we denote by $\prox_g$ its proximity operator defined by
\begin{equation}
\prox_g(x) = \arg\min_{w \in {\cal X}} \Bigl[ 
g(w)+\frac 12\|w-x\|^2 \Bigr] . 
\label{eq:prox}
\end{equation}

The ADMM+ is an instance of the primal dual algorithm recently 
proposed by V\~u\cite{vu2013splitting} and Condat~\cite{cond-13}, 
see also~\cite{he-yua-siam12}: 
\begin{theorem}[\!\!\cite{vu2013splitting,cond-13}] 
\label{vu} 
Given a Euclidean space $\cal E$, consider the minimization problem 
$\inf_{y\in{\cal E}} \bar f(y) + \bar g(y) + h(y)$ where 
$\bar g, h \in \Gamma_0(\cal E)$ and where $\bar f$ is convex and 
differentiable on $\cal E$ with an $L-$Lipschitz continuous gradient. 
Assume that the infimum is attained and that  
$0 \in \ri(\dom h - \dom \bar g)$. 
Let $\tau,\rho > 0$ be such that $\tau^{-1} - \rho^{-1} > L/2$, and consider 
the iterates 
\begin{subequations} 
\label{alg-vu} 
\begin{align}
\lambda^{k+1} &= \prox_{\rho^{-1} h^*}( \lambda^k + \rho^{-1} y^k) 
\label{vu-l} \\
y^{k+1} &= \prox_{\tau \bar g}( y^k - \tau\nabla \bar f(y^k) 
       - \tau (2\lambda^{k+1} -\lambda^k) ) . 
\label{vu-y} 
\end{align}
\end{subequations}  
Then for any initial value $(y^0, \lambda^0) \in {\cal E} \times {\cal E}$, 
the sequence $(y^k, \lambda^k)$ converges to a primal-dual point 
$(y^\star, \lambda^\star)$, \emph{i.e.}, a solution of the equation 
\begin{equation}
\label{pd-vu} 
\inf_{y\in{\cal E}} \bar f(y) + \bar g(y) + h(y) = 
- \inf_{\lambda \in\cal E} (\bar f + \bar g)^*(-\lambda) + h^*(\lambda).  
\end{equation} 
\end{theorem} 

\subsection{The case $f\equiv 0$ and the link with ADMM}

In the special case $f\equiv0$ and $\tau=\rho$, sequence $(u^k)_{k\in \mathbb{N}}$ coincides with $(z^k)_{k\in \mathbb{N}}$.
Then, ADMM+ boils down to the standard ADMM whose iterations are given by:
    \begin{align*}
z^{k+1} &= \argmin_{z\in {\cal Y}}\left[ h(z)+
                \frac 1{2\rho}\|z-Mx^{k}-\rho\lambda^k\|^2 \right]\\
\lambda^{k+1} &= \lambda^k + \rho^{-1}(Mx^{k}-z^{k+1}) \\ 
x^{k+1} &= \argmin_{x\in {\cal X}}\left[ g(x)
    + \frac1{2\rho} \|Mx-z^{k+1}+\rho\lambda^{k+1}\|^2 \right] . 
\end{align*}


\subsection{The case $h\equiv 0$ and the link with the Forward-Backward 
algorithm}

In the special case $h\equiv 0$ and $M=I$, it can be easily verified that $\lambda^k$ is null for all $k\geq 1$
and $u^k=x^k$.
Then, ADMM+ boils down to the standard Forward-Backward algorithm whose iterations are given by:
\begin{align*}
x^{k+1} &= \argmin_{x\in {\cal X}} g(x)+ \frac1{2\tau} \| x-( x^{k} - \tau \nabla f(x^k)  )\|^2 \\
&= \prox_{\tau g} ( x^{k} - \tau \nabla f(x^k)).
\end{align*}
One can remark that $\rho$ has disappeared thus it can be set as large
as wanted so the condition on stepsize $\tau$ from
Theorem~\ref{the:cv} boils down to $\tau < 2/L$.  Applications of this
algorithm with particular functions appear in well known learning
methods such as ISTA \cite{daubechies2004}.

\subsection{Comparison to the original V\~u-Condat algorithm}  

We emphasize the fact that ADMM+ is a variation on the V\~u-Condat algorithm.
The original V\~u-Condat algorithm is in general sufficient and, in many contexts, 
has even better properties than ADMM+ from an implementation point of view.
Indeed, whereas the V\~u-Condat algorithm handles the operator $M$ explicitly,
the step (\ref{adm+-x}) in ADMM+ can be delicate to implement in certain applications,
\emph{i.e.}, when $M$ has no convenient structure allowing to easily compute the $\arg\min$
(the same remark holds of course for ADMM which is a special case of ADMM+).

This potential drawback is however not an issue in many other scenarios where the structure of $M$
is such that  step~(\ref{adm+-x}) is affordable. In Section~\ref{sec:distop},
we shall provide such scenarios where ADMM+ is especially relevant.
In particular, ADMM+ is not only easy to implement but it is also provably convergent under weaker assumptions on the step sizes,
as compared to the original V\~u-Condat algorithm.

Also, the injectivity assumption on $M$ could be seen as restrictive at first glance. 
First, the latter assumption is in fact not needed when $f=0$ as noted above.
Second, it is trivially satisfied in the application scenarios which motivate this paper (see the next sections).

\color{black}


{
As alluded to in the introduction, the primal-dual algorithm
of~\cite{vu2013splitting,cond-13} can be geometrically described as a sequence
of Krasnosel'skii-Mann iterations applied to an $\alpha$-averaged operator. In
the next section, we briefly present these notions, proceed by introducing the
randomized coordinate descent version of these iterations, then state our
convergence result.}  

\section{Coordinate Descent}
\label{sec:cd}

\subsection{Averaged operators and the primal-dual algorithm} 
\label{subsec-avg} 

Let $\cH$ be a Euclidean space\footnote{We refer to \cite{livre-combettes} for an extension to Hilbert spaces.}.
For $0<\alpha\leq 1$, a mapping $\oT:\cH\to\cH$  is 
\emph{$\alpha$-averaged} if the following inequality holds for any $x,y$ in $\cH$:
$$
\|\oT x-\oT y\|^2\leq \|x-y\|^2 - \frac{1-\alpha}\alpha \|(\oI-\oT)x-(\oI-\oT)y\|^2\,.
$$
A 1-averaged operator is said \emph{non-expansive}. A $\frac 12$-averaged operator
is said \emph{firmly non-expansive}. The following Lemma can be found in {\cite[Proposition 5.15, pp.80]{livre-combettes}.

\begin{lemma}[Krasnosel'skii-Mann iterations]
\label{lem:krasno} 
  Assume that $\oT:\cH\to\cH$ is $\alpha$-averaged and that the set $\fix(\oT)$ of fixed points of $\oT$ is non-empty.
Consider a sequence $(\eta_k)_{k\in \NN}$ such that $0\leq \eta_k\leq 1/\alpha$ and $\sum_k\eta_k(1/\alpha-\eta_k)=\infty$.
For any $x^0\in \cH$, the sequence $(x^k)_{k\in \NN}$ recursively defined on $\cH$ by $x^{k+1} = x^k + \eta_k(\oT x^k-x^k)$ converges to some point in $\fix(\oT)$. 
\end{lemma}

On the product space ${\cal Y}\times {\cal Y}$, consider the operator 
$$
  \oV =
  \begin{pmatrix}
    \tau^{-1} \oI_{\cal Y} & \oI_{\cal Y} \\
    \oI_{\cal Y} & \rho \oI_{\cal Y}
  \end{pmatrix}
$$
{\color{black} where $\oI_{\cal Y}$ stands for the identity on ${\cal Y}\to{\cal Y}$.}
When $\tau^{-1} - \rho^{-1} > 0$, one can easily check that $\oV$ is positive
definite. In this case, we endow ${\cal Y}\times {\cal Y}$ with an inner 
product $\langle \,\cdot\,,\,\cdot\,\rangle_{\oV}$ defined as $\langle
\zeta,\varphi\rangle_{\oV} = \langle \zeta,\oV\varphi\rangle$ where 
$\langle \,\cdot\,,\,\cdot\,\rangle$ stands for the natural inner product on 
${\cal Y}\times{\cal Y}$. We denote by ${\cal H}_\oV$ the corresponding 
Euclidean space. \\ 
In association with Lemma~\ref{lem:krasno}, the following lemma is at the 
heart of the proof of Theorem~\ref{vu}: 

\begin{lemma}[\!\!\cite{vu2013splitting,cond-13}] 
\label{lem:cv-fbpd}
Let Assumptions~\ref{cvx}--\ref{hyp:M2} hold true. Assume that 
$\tau^{-1} - \rho^{-1} > L/2$. Let 
$(\lambda^{k+1}, y^{k+1} ) = \oT(\lambda^{k}, y^{k})$ where $\oT$ is the 
transformation described by Equations~\eqref{vu-l}--\eqref{vu-y}. Then 
$\oT$ is an $\alpha$-averaged operator on ${\cal H}_\oV$ with 
$\alpha = (2-\alpha_1)^{-1}$ and 
$\alpha_1 = (L/2)(\tau^{-1} - \rho^{-1})^{-1}$. 
\end{lemma}

Note that $\tau^{-1} - \rho^{-1} > L/2$ implies that $1 > \alpha_1 \geq 0$ and thus that  $\alpha$ verifies $\frac 12 \leq \alpha < 1$ which matches the definition of $\alpha$-averaged operators.

\subsection{Randomized Krasnosel'skii Mann Iterations}
\label{sec:randomKM} 
Consider the space $\cH=\cH_1\times\dots\times\cH_J$ for some 
integer $J\geq 1$ where for any $j$, $\cH_j$ is a Euclidean space. 
Assume that $\cH$ is equipped with the scalar product 
$\langle x,y\rangle = \sum_{j=1}^J \langle x_j,y_j\rangle_{\cH_j}$
where $\langle\,\cdot\,,\,\cdot\,\rangle_{\cH_j}$ is the scalar product in $\cH_j$.
For $j\in\{1,\dots,J\}$, we denote by $\oT_j:\cH\to\cH_j$ the components of the output of operator $\oT:\cH\to\cH$ corresponding to $\cH_j$, we thus have
$\oT x = (\oT_1 x,\dots,\oT_J x)$. 
We denote by $2^{\cal J}$ the power set of ${\cal J}=\{1,\dots,J\}$.
For any $\kappa\in 2^{\cal J}$,
we define the operator $\hat\oT^{(\kappa)}:\cH\to\cH$ by
$\hat \oT^{(\kappa)}_j x = \oT_j x$ if $j\in \kappa$ and 
$\hat \oT^{(\kappa)}_j x =x_j$ otherwise.
On some probability space $(\Omega,{\mathcal F}, {\mathbb P})$, we introduce a
random i.i.d. sequence $(\xi^k)_{k\in\mathbb{N}^*}$ such that $\xi^k : \Omega
\to 2^{\cal J}$ \emph{i.e.} $\xi^k(\omega)$ is a subset of $\cal J$.  We assume
that the following holds:
\begin{equation}
  \label{eq:2}
  \forall j\in{\cal J},\,\exists \kappa\in 2^{\cal J} \text{ s.t.}\, j\in\kappa\text{ and }{\mathbb P}(\xi_1 = \kappa) >0\,. 
\end{equation}

Let  $\oT$ be an $\alpha$-averaged operator, instead of considering the iterates 
$x^{k+1} = x^k + \eta_k(\oT x^k - x^k)$, 
we are now interested in a stochastic \emph{coordinate descent} version 
of this algorithm that consists in iterates of the type 
$x^{k+1} = x^k + \eta_k(\hat\oT^{(\xi^{k+1})} x^k - x^k)$. The proof of Theorem~\ref{the:randomKM}
is provided in Appendix~\ref{sec:proofrandomKM}.

\begin{theorem}
\label{the:randomKM}
  Let $\oT:\cH\to\cH$ be $\alpha$-averaged and $\fix(\oT)\neq\emptyset$.
Assume that for all $k$, the sequence $(\eta_k)_{k\in \NN}$ satisfies
$$
0<\liminf_k\eta_k\leq \limsup_k\eta_k<\frac 1\alpha\,.
$$
Let $(\xi^k)_{k\in \NN^*}$ be a random i.i.d. sequence on $2^{\cal J}$ such 
that Condition~\eqref{eq:2} holds. Then, for any deterministic initial value
$x_0$, the iterated sequence
\begin{equation}
x^{k+1} = x^k + \eta_k(\hat\oT^{(\xi^{k+1})} x^k-x^k)\label{eq:3}
\end{equation}
converges almost surely to a random variable supported by $\fix(T)$.
\end{theorem}
\begin{remark}
At the time of the writing the paper, the work \cite{pesquet-stochastic} was
brought to our knowledge.  A result similar to Theorem~\ref{the:randomKM} is
presented {\color{black} in the framework of Hilbert spaces, 
random summable errors (dealt with by relying on the
notion of quasi-F\'ejer monotonicity) and multiple blocks.}  
{The proof of~\cite{pesquet-stochastic} devoted to this result
relies on the same idea as the one developed in~\cite{iut-cdc13} and presented 
above. Distributed asynchronous implementations are not considered 
in~\cite{pesquet-stochastic}}.
\end{remark}


By Lemma~\ref{lem:cv-fbpd}, ADMM+ iterates are generated by the action of
an $\alpha$-averaged operator. Theorem~\ref{the:randomKM} shows then that a 
stochastic coordinate descent version of {any $\alpha$-averaged operator
 converges towards a primal-dual point. In Theorem~\ref{th:dapd} below, we apply 
this result to the operator related to ADMM+, 
and develop an asynchronous version of ADMM+ in the context where it is distributed on a graph.}

\section{Distributed Optimization}
\label{sec:distop}

Consider a set of $N > 1$ computing agents that cooperate to solve 
the minimization problem
\begin{equation}
\inf_{x\in \cal X} \sum_{n=1}^N (f_n(x) + g_n(x))\label{eq:pb-distrib}
\end{equation}
where $f_n$ and $g_n$ are two private functions available at Agent $n$. 
We make here the following assumption: 
\begin{assumption}
\label{hyp:pb-dist}
For each $n = 1,...,N$,
  \begin{enumerate}[(i)]
 \item $f_n$ is a convex differentiable function on ${\cal X}$, and its gradient 
       $\nabla f_n$ is $\bar L$-Lipschitz continuous on ${\cal X}$ for some $\bar L\geq 0$.
 \item $g_n\in \Gamma_0({\cal X})$.
\item The infimum of Problem (\ref{eq:pb-distrib}) is attained.
\item $\cap_{n=1}^N \ri \dom g_n\neq \emptyset$.
  \end{enumerate}
\end{assumption}
Our purpose is to design a random distributed (or decentralized) iterative
algorithm where, at a each iteration, each active agent updates a local
estimate in the parameter space ${\cal X}$ based on the sole knowledge of its
private functions and on information it received from its neighbors through
some communication network. Eventually, the local estimates will converge to a
common consensus value which is a minimizer of the aggregate function of
Problem~(\ref{eq:pb-distrib}).

Instances of this problem appear in learning applications
where massive training data sets are distributed over a network and processed 
by distinct machines \cite{for-can-gia-jmlr10,agarwal2011reliable}, 
in resource allocation problems for communication networks 
\cite{bia-jak-TAC13}, or in statistical estimation problems by
sensor networks \cite{ram-vee-ned-tac10,bia-for-hac-IT13}. 

\subsection{Network Model and Problem Formulation}

{
To help the reader, the notations that will be introduced progressively are
summarized in the following table. 
\begin{center}
\begin{tabular}{|lcl|}  
\hline
$d_n$ &:& Degree of node (agent) $n$, \\ 
$\epsilon = \{n,m\}$ &:& Graph edge between $n$ and $m$, \\ 
$E$ &:& Set of graph edges, \\
$f(x) = \sum f_n(x_n)$ &:& Differentiable term in obj.~fct., \\ 
$g(x) = \sum g_n(x_n)$ &:& Other $\Gamma_0$ term in obj.~fct., \\ 
$h$ &:& Consensus ensuring function, \\ 
$\lambda^k$ &:& $= ( (\lambda_\epsilon^k(n), \lambda_\epsilon^k(m) ) 
               )_{\epsilon = \{n,m\}\in E}$ \\ 
& & ${\mathcal X}^{2|E|}$ vector of dual variables, \\
& & $\lambda_\epsilon^k(n)$ is updated by Agent $n$, \\ 
$n \sim m$ &:& Stands for $\{n,m\} \in E$, \\  
Subscript $n$ &:& Agent number, \\
Superscript $k$ &:& Time index, \\ 
$V = \{1,\ldots, N\}$ &:& Set of graph nodes (agents), \\ 
$x^k = (x^k_n)_{n\in V}$ &:& ${\mathcal X}^{N}$ vector of primal 
                            variables, \\  
& & updated by Eq.~\eqref{adm+-x},  \\ 
$z^k = ( (\bar z_\epsilon^k, \bar z_\epsilon^k ) )_{\epsilon \in E}$ &:& 
${\mathcal X}^{2|E|}$ vector given by Eq.~\eqref{adm+-z}. \\ 
\hline
\end{tabular} 
\end{center} 
} 
We represent the network as an undirected graph $G = (V,E)$ where 
$V=\{1,\dots,N\}$ is the set of agents/nodes and $E$ is the set of edges. 
Representing an edge by a set $\{n, m\}$ with $n,m \in V$, we 
write $m\sim n$ whenever $\{n,m\}\in E$.  Practically, $n\sim m$
means that Agents $n$ and $m$ can communicate with each other. 
\begin{assumption}\label{hyp:connected}
  $G$ is connected and has no self-loops ($n\neq m$ for all $\{ n,m\} \in E$). 
\end{assumption}

Let us introduce some notation.  For any $x\in {\cal X}^{N}$, we
denote by $x_n$ the $n^\textrm{th}$ component of $x$, \emph{i.e.}, $x= (x_n)_{n\in {V}}$.
We introduce the functions $f$ and $g$ on ${\cal X}^{N}\to
(-\infty,+\infty]$ as 
$f(x) = \sum_{n\in V}f_n(x_n)$ and $g(x) = \sum_{n\in V}g_n(x_n)$. 
Clearly, Problem~\eqref{eq:pb-distrib} is equivalent to the minimization of
$f(x)+g(x)$ under the constraint that all components of $x$ are equal. 
\color{black}
Here, one can rephrase the optimization problem~\eqref{eq:pb-distrib} as 
\[
\min_{x\in {\cal X}^N} \sum_{n=1}^N (f_n(x_n) + g_n(x_n)) + \iota_{\cal C}(x)
\] 
where $\iota_A$ is the indicator function of a set $A$ (null on
$A$ and equal to $+\infty$ outside this set), and $\cal C$ is the space of
vectors $x\in {\cal X}^N$ such that $x_1=\cdots=x_N$.
This problem is an instance of Problem~\eqref{pb} where $h=\iota_{\cal C}$ and 
$M$ as the identity operator. 
However, simply setting $h=\iota_{\cal C}$ and $M$ as the identity would not lead 
to a distributed algorithm. Loosely speaking, we must define $h$ and $M$ in such a 
way that it encodes the communication graph. \color{black}
Our goal will be to ensure global consensus through local consensus over every 
edge of the graph.

For any $\epsilon\in E$, say $\epsilon= \{n,m\}$, we define 
the linear operator $M_\epsilon:{\cal X}^{N}\to {\cal X}^2$
as $M_\epsilon (x) = (x_n,x_{m})$ assuming $n < m$ to avoid any
ambiguity on the definition of $M$. We construct the linear operator 
$M:{\cal X}^{N}\to {\cal Y}\triangleq {\cal X}^{2|E|}$ as
$Mx = \left(M_\epsilon(x)\right)_{\epsilon\in E}$ where we assume some 
(irrelevant) ordering on the edges. Any vector 
$y \in {\cal Y}$ will be written as $y=(y_\epsilon)_{\epsilon\in E}$ where, 
writing 
$\epsilon= \{n,m\}\in E$, the component $y_\epsilon$ will be represented 
by the couple $y_\epsilon = (y_{\epsilon}(n),y_{\epsilon}(m))$ with $n < m$. 
Note that this notation is abusive since it tends to indicate that
$y_\epsilon$ has more than two components.  
However, it will turn out to be convenient in the sequel. 
We also introduce the subspace of ${\cal X}^2$ defined as 
${\cal C}_2=\{(x,x) : x\in {\cal X}\}$. Finally, we define 
$h:{\cal Y}\to \mathbb (-\infty,+\infty]$ as 
\begin{equation}
\label{def:h} 
h(y) = \sum_{\epsilon\in E} \iota_{{\cal C}_2}(y_\epsilon)\,.
\end{equation} 
We consider the following problem:
\begin{equation}
  \label{eq:pb-od}
  \min_{x\in {\cal X}^{N}} f(x)+g(x)+h(Mx)\ .
\end{equation}
\begin{lemma}
\label{consensus} 
Let Assumption \ref{hyp:connected} hold true. 
The minimizers of~\eqref{eq:pb-od} are the tuples $(x^\star,\cdots,x^\star)$
where $x^\star$ is any minimizer of~\eqref{eq:pb-distrib}.
\end{lemma}
\begin{proof} Assume that Problem~\eqref{eq:pb-od} has a minimizer 
$\underline x = ( x_1, \ldots, x_{N})$. Then 
\[
h(M\underline x) = 
\sum_{\epsilon = \{n,m\}\in E} \iota_{{\cal C}_2}( ( x_n, x_m ))  = 0 . 
\]
Since the graph $G$ is connected, this equation is satisfied if and only if
$\underline x = (x^\star, \ldots, x^\star)$ for some $x^\star \in \cal X$. 
The result follows. 
\end{proof}

\subsection{Instantiating ADMM+}

We now apply ADMM+ to solve the problem~\eqref{eq:pb-od}. 
Since the newly defined function $h$ is separable with respect to the 
$(y_\epsilon)_{\epsilon\in E}$, we get 
\[
\prox_{\rho h}(y) = 
( \prox_{\rho \iota_{{\cal C}_2}}(y_\epsilon) )_{\epsilon \in E} 
= 
\Bigl( (\bar y_{\epsilon},\bar y_{\epsilon}) 
                     \Bigr)_{\epsilon \in E}
\]
where $\bar y_{\epsilon} = (y_{\epsilon}(n)+y_{\epsilon}(m))/2$ if $\epsilon = \{n,m\} $. 
With this at hand, the update equation~\eqref{adm+-z} of ADMM+ is
written as $z^{k+1} = 
( (\bar z^{k+1}_{\epsilon},\bar z^{k+1}_{\epsilon}) )_{\epsilon \in E}$ 
where $\bar z^{k+1}_{\epsilon} = 
(x_n^{k} + x_m^{k})/2 + 
    \rho(\lambda^k_{\epsilon}(n) + \lambda^k_{\epsilon}(m))/2$
for any $\epsilon= \{n,m\}\in E$.
Plugging this equality into Eq.~\eqref{adm+-l}, it can be seen that 
$\lambda^k_{\epsilon}(n) = -\lambda^k_{\epsilon}(m)$.
Therefore, $\bar z_\epsilon^{k+1}=(x_n^{k} + x_m^{k})/2$ for any $k\geq 1$.
Moreover, $\lambda^{k+1}_{\epsilon}(n) = \lambda^k_{\epsilon}(n)+
(x_n^{k} - x_m^{k})/(2\rho)$. \\ 
Let us now instantiate Equations~\eqref{adm+-u} and~\eqref{adm+-x}. 
Observe that the $n^\textrm{th}$ component of the vector $M^* M x$ coincides 
with $d_n x_n$ where $d_n$ is the degree (\emph{i.e.}, the number of neighbors)
of node~$n$. 
From Eq.~\eqref{adm+-x}, the $n^{\text{th}}$ component of $x^{k+1}$ is written 
\[
x^{k+1}_n = \prox_{\tau g_n/d_n}\Bigl[ 
 \frac{(M^* (u^{k+1} - \tau \lambda^{k+1}))_n  
   - \tau\nabla f_n(x^k_n)}{d_n} \Bigr]
\]
where for any $y \in \cal Y$, 
\[
(M^*y)_n = \sum_{m : \{n,m\} \in E} y_{\{n,m\}}(n)
\]
is the $n^\textrm{th}$ component of $M^*y \in {\cal X}^{N}$.  
Plugging Eq.~\eqref{adm+-u} together with the expressions of 
$\bar z_{\{n,m\}}^{k+1}$ and $\lambda^{k+1}_{\{n,m\}}(n)$ in the argument of 
$\prox_{\tau g_n/d_n}$, we get after a small calculation 
\begin{multline*} 
x^{k+1}_n = \prox_{\tau g_n/d_n}\Bigl[ (1 - \tau\rho^{-1}) x_n^k 
- \frac{\tau}{d_n} \nabla f_n(x_n^k)  \\
 + \frac{\tau}{d_n} \sum_{m:\{n,m\} \in E} (\rho^{-1} x_m^k 
- \lambda_{\{n,m\}}^k(n)) \Bigr] .
\end{multline*}
{The \emph{Distributed ADMM+} (DADMM+) algorithm is described by the 
following procedure:} 
\smallskip
\begin{breakbox}
\noindent {\bf DADMM+} \\
Initialization: $(x^0,\lambda^0)$ s.t. $\lambda^0_{\{n,m\}}(n) = 
- \lambda^0_{\{n,m\}}(m)$ for all $m\sim n$. \\
Do 
\begin{itemize} 
\item For all $n \in V$, Agent $n$ {has in its memory the variables
$x_n^k$, $\{\lambda_{\{n,m\}}^k(n) \}_{m\sim n}$, and $\{ x^k_m \}_{m\sim n}$.
It} performs the following operations: 
\begin{itemize} 
\item For all $m\sim n$, do 
\[
\lambda^{k+1}_{\{n,m\}}(n) = \lambda^k_{\{n,m\}}(n) 
       + \frac{x_n^{k} - x_m^{k}}{2\rho} ,  
\]
\item 
$\displaystyle{
x^{k+1}_n = \prox_{\tau g_n/d_n}\Bigl[ (1 - \tau\rho^{-1}) x_n^k 
- \frac{\tau}{d_n} \nabla f_n(x_n^k)}$ \\ 
$\displaystyle{\ \ \ \ \ \ \ \ \ \ 
 + \ \frac{\tau}{d_n} \sum_{m:\{n,m\} \in E} (\rho^{-1} x_m^k 
- \lambda_{\{n,m\}}^k(n)) \Bigr]}$. 
\end{itemize} 
\item For all $n \in V$, Agent~$n$ sends the parameter $x^{k+1}_n$ to its 
neighbors, 
\item Increment $k$.  
\end{itemize}
\end{breakbox}
\smallskip


\medskip

The proof of the following result is provided in Appendix~\ref{app:proof-DADMM+}.
\begin{theorem}
\label{the:DADMM+}
Let Assumptions~\ref{hyp:pb-dist} and \ref{hyp:connected} hold true.
Assume that {\color{black}
  \begin{equation}
  \tau^{-1}-\rho^{-1} > \frac{\bar L}{2d_{\mathrm{min}}}
\label{eq:cond-pas-DADMM+}
\end{equation}
where $d_{\mathrm{min}}$ is the minimum of the nodes' degrees in the graph $G$.}
For any initial value $(x^0,\lambda^0)$, let $(x^{k})_{k \in  \mathbb{N}}$ be 
the sequence produced by the Distributed ADMM+. 
Then there exists a minimizer $x^\star$ of Problem~\eqref{eq:pb-distrib} such 
that for all $n\in V$, $(x_n^k)_{k\in \mathbb{N}}$ converges to $x^\star$.
\end{theorem}

\subsection{A Distributed Asynchronous Primal Dual Algorithm}

\color{black}
In the distributed \emph{synchronous} case, at each clock tick, a central scheduler
activates all the nodes of the network simultaneously and monitors the
communications that take place between these nodes once they have finished
their $\text{prox}(\cdot)$ and gradient operations. The meaning we give to ``distributed
asynchronous algorithm'' is that there is no central scheduler and that any
node can wake up randomly at any moment independently of the other nodes. This
mode of operation brings clear advantages in terms of complexity and
flexibility. \color{black}

The proposed \emph{Distributed Asynchronous Primal Dual} algorithm (DAPD) is
obtained by applying the randomized coordinate descent on the above algorithm.
As opposed to the latter, the resulting algorithm has the following attractive
property: at each iteration, a single agent, or possibly a \emph{subset} of
agents chosen at random, are activated.  More formally, let $(\xi^k)_{k\in \mathbb{N}}$ be a
sequence of i.i.d. random variables valued in $2^V$. The value taken by $\xi^k$
represents the agents that will be activated and perform a $\prox$ on their $x$
variable at moment $k$. The asynchronous algorithm goes as follows: 

\smallskip

\begin{breakbox}
\noindent  {\bf DAPD Algorithm}: \\
Initialization: $(x^0, \lambda^0)$. \\
Do 
\begin{itemize} 
\item Select a random set of agents $\xi^{k+1} = {\cal A}$, 
\item For all $n\in {\cal A}$, Agent $n$ performs the following operations:  
  \begin{itemize}
  \item For all $m \sim n$, do 
 \begin{align*}
  & \lambda_{\{n,m\}}^{k+1}(n) = 
    \frac{\lambda_{\{n,m\}}^k(n) - \lambda_{\{n,m\}}^k(m)}{2} \\
      & \ \ \ \ \ \ \ \ \ \ \ \ \ \ \ \ \ 
         + \frac{x^k_n - x^k_m}{2\rho} \, , 
 \end{align*} 
  \item $\displaystyle{x_n^{k+1} = \prox_{\tau g_n/d_n} \Bigl[ 
    (1 - \tau\rho^{-1}) x^k_n - \frac{\tau}{d_n} \nabla f_n(x_n^k)}$ \\
 $\displaystyle{ 
     \ \ \ \ \ \ \ \ \ \ \ \ \ \ \ \ \ 
    + \frac{\tau}{d_n} 
    \sum_{m\sim n} ( \rho^{-1} x^k_m + \lambda^k_{\{n,m\}}(m) ) \Bigr]}$,  
  \item For all $m\sim n$, send 
      $\{ x_n^{k+1}, \lambda_{\{n,m\}}^{k+1}(n) \}$ to Neighbor $m$. 
  \end{itemize} 
\item For all agents $n \not\in{\cal A}$, $x^{k+1}_n = x^k_n$, and 
$\lambda^{k+1}_{\{n,m\}}(n) = \lambda^{k}_{\{n,m\}}(n)$ for all $m\sim n$. 
\item Increment $k$.   
\end{itemize} 
\end{breakbox}
\smallskip

\begin{assumption}
\label{hyp:do-iid}
The collections of sets $\{{\cal A}_1, {\cal A}_2, \ldots\}$ such that 
${\mathbb P}[\xi^1 = {\cal A}_i ]$ is positive satisfies 
$\bigcup {\cal A}_i = V$. 
\end{assumption}
In other words, any agent is selected with a positive probability. 
The following theorem is proven in Appendix~\ref{prf:dapd}. 
\begin{theorem}
\label{th:dapd}
Let Assumptions~\ref{hyp:pb-dist},~\ref{hyp:connected}, and~\ref{hyp:do-iid} 
hold true. Assume that condition~(\ref{eq:cond-pas-DADMM+}) holds true.
Let $(x_n^{k+1})_{n\in V}$ be the output of the DAPD algorithm.
For any initial value $(x^0,\lambda^0)$, the sequences $x_1^k, \ldots, 
x_{N}^k$ converge almost surely as $k\to\infty$ to a random variable 
$x^\star$ supported by the set of minimizers of Problem~\eqref{eq:pb-distrib}.
\end{theorem}

{\color{black} 
Before turning to the numerical illustrations, we note that the 
very recent paper~\cite{pes-rep-(sub)14} also deals with asynchronous 
primal-dual distributed algorithms by relying on the idea of random coordinate
descent}.

\section{Numerical illustrations}
\label{sec:num}

We address the problem of the so called $\ell_2$-regularized logistic
regression. Denoting by $m$ the number of observations and by $p$ the number
of features, our optimization problem is written 
\[
\min_{ \mathbf{x}\in\mathbb{R}^p}  \frac{1}{m} \sum_{t=1}^m \log\left( 1 + \mathrm{e}^{-y_t \mathbf{a}_t^\mathrm{T} \mathbf{x} } \right) + 
\mu \| \mathbf{x} \|^2  
\] 
where the $(y_t)_{t=1}^m$ are in $\{-1,+1\}$, the $(\mathbf{a}_t)_{t=1}^m$ are
in $\mathbb{R}^p$, and $\mu>0$ is a scalar. \\ 
We consider the case where the dataset is scattered
over a network. Indeed, massive data sets are often distributed on different
physical machines communicating together by means of an interconnection network
\cite[Chap.~2.5]{rauber2013parallel} and many algorithms have been implemented
for independent threads or processes running on distant cores, closer to the
data (see \emph{e.g.} \cite{tsianos2014efficient}, \cite{boyd2011distributed}
for MapReduce implementation of ADMM, \cite{jaggi2014communication} for Spark
implementation). 
Formally, denoting by $\{\mathcal{B}_n\}_{n=1}^N$ a partition of 
$\{1,\dots,m\}$, we assume that Agent~$n$ holds in its memory the data in
$\mathcal{B}_n$. Denoting by $G=(V,E)$ the graph that represents the 
connections between the agents, the regularized logistic regression problem is 
written in its distributed form as 
\begin{align*}
\nonumber \min_{ \mathbf{x} \in\mathbb{R}^{Np}} \sum_{n=1}^N  &\left( \sum_{t\in \mathcal{B}_n} \frac{1}{m} \log\left( 1 + \mathrm{e}^{-y_t \mathbf{a}_t^\mathrm{T} \mathbf{x}_n } \right) + \frac{\mu}{2N} \|  \mathbf{x}_n \|^2_2 \right)\\
 &+ \sum_{\epsilon\in E} \iota_{{\cal C}_2}(y_\epsilon). 
\end{align*}
Clearly, this is an instance of Problem~\eqref{eq:pb-od}. 

Our simulations will be performed on the following classical datasets:

\begin{center}
\begin{tabular}{|l|c|c|c|}
\hline
name & m  & p & density \\
\hline
\texttt{covtype} & $581 012$ & $54$ & dense \\
\hline
\texttt{alpha} & $500 000$  & $500$ & dense \\
\hline
\texttt{realsim} & $72 309$  & $20 958$ & sparse \\
\hline
\texttt{rcv1} & $20 242$  & $ 47 236 $ & sparse \\
\hline
\end{tabular}
\end{center}

The datasets \texttt{covtype}, \texttt{realsim}, and \texttt{rcv1} are taken
from the LIBSVM
website\footnote{\url{http://www.csie.ntu.edu.tw/~cjlin/libsvm/}} and
\texttt{alpha}  was from the Pascal 2008 Large Scale Learning
challenge\footnote{\url{http://largescale.ml.tu-berlin.de}}. We preprocessed
the dense datasets so that each feature has zero mean and unit variance. 
{The global Lipschitz constant for the gradient of the logistic function
was estimated by its classical upper bound $\hat{L} = 0.25 \max_{n=1,...,N} \| \mathbf{a}_n \|_2^2$.}
Finally, the regularization parameter $\mu$ was set to $10^{-4}$.

In our simulations, we also compared the DAPD algorithm presented in this paper
with some known algorithms that lend themselves to a distributed
implementation. These are:

\begin{itemize}
\item[-] DGD: the synchronous distributed algorithm \cite{ned-ozd-tac09}. Here, each agent performs a gradient descent then exchanges with its neighbors according to the Metropolis rule.

\item[-] ABG: the asynchronous broadcast gradient \cite{nedic2011}. In this
setup, one agent wakes up and sends its information to its neighbors. Any of 
these neighbors replaces its current value with the mean of this
value and the received value then performs a gradient descent.

\item[-] PWG: the pairwise gossip gradient \cite{tsitsiklis:bertsekas:athans:tac-1986,boyd2006randomized}. In this setup, one agent wakes up,
and selects one neighbor. Then each of the two agents performs a gradient 
descent, then exchanges and replaces its value by taking the mean between the 
former and the received value.
\end{itemize}

For the DGD, the ABG, and the PWG, the stepsizes have been taken decreasing as
$\gamma_0/k^{0.75}$. The other parameters (including $\gamma_0$) were chosen
automatically in sets of the form $ parameter_{theory} \times 10^i , i=\{1,..,10\}$  
{ (where $ parameter_{theory} $ is computed from the best theoretic bound 
with Lipschitz constant estimate $\hat{L}$) by running in parallel multiple 
instances of the algorithm with the different constants over $50$ iterations 
and choosing the constant giving the lowest functional cost.}

Whereas DAPD can allow for multiple agents to wake up at each iteration, we
considered only the single active agent case as it does not change much the
practical implementation. It is thus underperforming compared to a multiple
awaking agents scenario. Similarly to the previous algorithms, the stepsizes of 
DAPD have been chosen automatically
in sets of the form $ parameter_{theory} \times 10^i$ for $\tau$ and $\rho =
2\tau$ for fairness in terms of number of step sizes explored. 

The (total) functional cost was evaluated with the value at agent $1$ (the
agents are indistinguishable from a network point of view) and plotted versus
the number of local gradients used.

In Figure~\ref{fig:distopt}, we plot the $\ell_2$-regularized logistic cost at
some agent versus the number of local gradients used. We solved this problem
for each dataset on a $10\times 10$ 2D toroidal grid ($100$ agents) by
assigning the same number of observations per agent. We observe that the DAPD
is significantly faster than the other stochastic gradient methods. Finally,
we also remark that the quantity of information exchanged per iteration for
DAPD is roughly a vector of length shorter than $2Np$ ($8p$ with our graph)
which means that the number of transmissions is in general quite small compared
to the size of the whole dataset (roughly $Tp$).

In Figure~\ref{fig:distopt2}, we plot the same quantities for the \texttt{rcv1} dataset but now the same number of observations are dispatched over i) a $5 \times 5$ toroidal grid ($25$ agents) and ii) a $50$-nodes complete network.
 
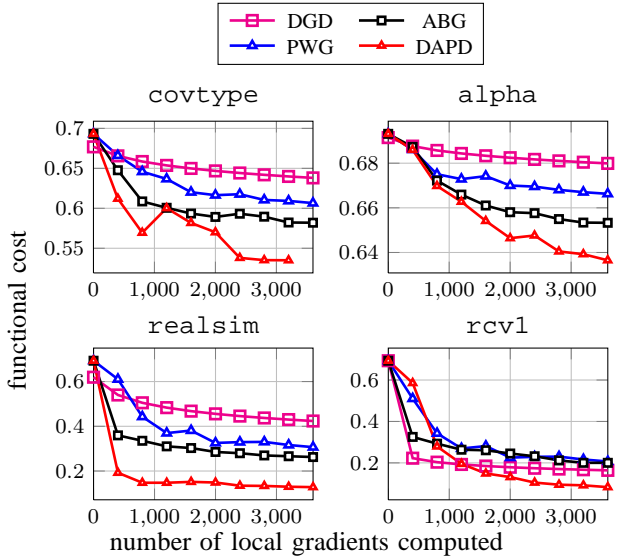
\begin{figure}[h!]
  \centering
  \begin{tikzpicture}
    \begin{groupplot}[group style={group size= 2 by 2},height=0.6\figureheight,width=0.75\figurewidth]

        \nextgroupplot[
title={\texttt{covtype}},
title style={at={(axis description cs:0.5,0.93)},anchor=south},
xmin=0,
xmax=3600,
xmajorgrids,
ymajorgrids,
]

\addplot [
color=magenta,
line width=1.0pt,
mark size=2.0pt,
mark=square,
mark options={solid,,fill=white,draw=magenta},
]
coordinates{
(  000.0  , 0.6768816898)
(  400.0  , 0.6656072215)
(  800.0  , 0.6585250881)
( 1200.0  , 0.6536536890)
( 1600.0  , 0.6498842235)
( 2000.0  , 0.6467908035)
( 2400.0  , 0.6441630916)
( 2800.0  , 0.6418777178)
( 3200.0  , 0.6398551909)
( 3600.0  , 0.6380410809)
};\label{plots:dgd}

\addplot [
color=blue,
line width=1.0pt,
mark size=1.8pt,
mark=triangle*,
mark options={solid,,fill=white,draw=blue},
]
coordinates{
(  000.0  , 	 0.6931471806 )
(  400.0  , 	 0.6659895975 )
(  800.0  , 	 0.6458020606 )
( 1200.0  , 	 0.6367368331 )
( 1600.0  , 	 0.6201345730 )
( 2000.0  , 	 0.6164225854 )
( 2400.0  , 	 0.6176081929 )
( 2800.0  , 	 0.6104653052 )
( 3200.0  , 	 0.6092406151 )
( 3600.0  , 	 0.6065067719 )
};\label{plots:pwg}

\addplot [
color=black,
line width=1.0pt,
mark size=1.4pt,
mark=square*,
mark options={solid,,fill=white,draw=black},
]
coordinates{
(  000.0  ,	 0.6931471806)
(  400.0  ,	 0.6476240299)
(  800.0  ,	 0.6085797006)
( 1200.0  ,	 0.6006191503)
( 1600.0  ,	 0.5934456252)
( 2000.0  ,	 0.5892014413)
( 2400.0  ,	 0.5932097067)
( 2800.0  ,	 0.5895189511)
( 3200.0  ,	 0.5822150958)
( 3600.0  ,	 0.5819610343)
};\label{plots:abg}

\addplot [
color=red,
line width=1.0pt,
mark size=1.5pt,
mark=triangle*,
mark options={solid,fill=white,draw=red},
]
coordinates{
(  000.0  , 0.6931471806 )
(  400.0  , 0.6120611939 )
(  800.0  , 0.5694247381 )
( 1200.0  , 0.6001922446 )
( 1600.0  , 0.5816862113 )
( 2000.0  , 0.5700125227 )
( 2400.0  , 0.5379502820 )
( 2800.0  , 0.53512820 )
( 3200.0  , 0.53495820 )
};\label{plots:dapd}

\coordinate (top) at (rel axis cs:0,1);

        \nextgroupplot[
title={\texttt{alpha}},
title style={at={(axis description cs:0.5,0.93)},anchor=south},
xmin=0,
xmax=3600,
xmajorgrids,
ymajorgrids]

\addplot [
color=magenta,
line width=1.0pt,
mark size=2.0pt,
mark=square,
mark options={solid,,fill=white,draw=magenta},
]
coordinates{
(  000.0  ,		 0.6914727355)
(  400.0  ,		 0.6876510153)
(  800.0  ,		 0.6857113973)
( 1200.0  ,		 0.6843698687)
( 1600.0  ,		 0.6833210588)
( 2000.0  ,		 0.6824497368)
( 2400.0  ,		 0.6816994781)
( 2800.0  ,		 0.6810378898)
( 3200.0  ,		 0.6804444204)
( 3600.0  ,		 0.6799051228)
};\label{plots:dgd}

\addplot [
color=blue,
line width=1.0pt,
mark size=1.8pt,
mark=triangle*,
mark options={solid,,fill=white,draw=blue},
]
coordinates{
(  000.0  ,	 0.6931471806)
(  400.0  ,	 0.6860664768)
(  800.0  ,	 0.6750560738)
( 1200.0  ,	 0.6728722219)
( 1600.0  ,	 0.6741991754)
( 2000.0  ,	 0.6700020756)
( 2400.0  ,	 0.6695080892)
( 2800.0  ,	 0.6680600906)
( 3200.0  ,	 0.6670197472)
( 3600.0  , 	 0.6662792512)
};\label{plots:pwg}

\addplot [
color=black,
line width=1.0pt,
mark size=1.4pt,
mark=square*,
mark options={solid,,fill=white,draw=black},
]
coordinates{
(  000.0  ,	 	 0.6931471806)
(  400.0  ,	 	 0.6872864521)
(  800.0  ,	 	 0.6722229884)
( 1200.0  ,	 	 0.6659334153)
( 1600.0  ,	 	 0.6610515895)
( 2000.0  ,	 	 0.6580416051)
( 2400.0  ,	 	 0.6576168911)
( 2800.0  ,	 	 0.6549667443)
( 3200.0  ,	 	 0.6533963330)
( 3600.0  ,	 	 0.6532992500)
};\label{plots:abg}

\addplot [
color=red,
line width=1.0pt,
mark size=1.5pt,
mark=triangle*,
mark options={solid,fill=white,draw=red},
]
coordinates{
(  000.0  , 0.6931471806 )
(  400.0  , 0.6859895975 )
(  800.0  , 0.6698020606 )
( 1200.0  , 0.6627368331 )
( 1600.0  , 0.6541345730 )
( 2000.0  , 0.6464225854 )
( 2400.0  , 0.6476081929 )
( 2800.0  , 0.6404653052 )
( 3200.0  , 0.6392406151 )
( 3600.0  , 0.6365067719 )
};\label{plots:dapd}

        \nextgroupplot[
title={\texttt{realsim}},
title style={at={(axis description cs:0.5,0.93)},anchor=south},
xmin=0,
xmax=3600,
xmajorgrids,
ymajorgrids]

\addplot [
color=magenta,
line width=1.0pt,
mark size=2.0pt,
mark=square,
mark options={solid,,fill=white,draw=magenta},
]
coordinates{

(  000.0  , 	 0.6194087961)
(  400.0  , 	 0.5399587272)
(  800.0  , 	 0.5050685681)
( 1200.0  , 	 0.4833699836)
( 1600.0  , 	 0.4677423385)
( 2000.0  , 	 0.4555617758)
( 2400.0  , 	 0.4456067182)
( 2800.0  , 	 0.4372083917)
( 3200.0  , 	 0.4299595500)
( 3600.0  , 	 0.4235934746)
};\label{plots:dgd}

\addplot [
color=blue,
line width=1.0pt,
mark size=1.8pt,
mark=triangle*,
mark options={solid,,fill=white,draw=blue},
]
coordinates{
(  000.0  , 0.6931471806)
(  400.0  , 0.6101871034)
(  800.0  , 0.4427174885)
( 1200.0  , 0.3695722672)
( 1600.0  , 0.3810675788)
( 2000.0  , 0.3255061845)
( 2400.0  , 0.3292482741)
( 2800.0  , 0.3301961952)
( 3200.0  , 0.3163317944)
( 3600.0  , 0.3067466149)
};\label{plots:pwg}

\addplot [
color=black,
line width=1.0pt,
mark size=1.4pt,
mark=square*,
mark options={solid,,fill=white,draw=black},
]
coordinates{
(  000.0  ,	 0.6931471806)
(  400.0  ,	 0.3594464805)
(  800.0  ,	 0.3352175387)
( 1200.0  ,	 0.3102932490)
( 1600.0  ,	 0.3027204431)
( 2000.0  ,	 0.2852179787)
( 2400.0  ,	 0.2796104298)
( 2800.0  ,	 0.2696949443)
( 3200.0  ,	 0.2661087170)
( 3600.0  ,	 0.2625060894)
};\label{plots:abg}

\addplot [
color=red,
line width=1.0pt,
mark size=1.5pt,
mark=triangle*,
mark options={solid,fill=white,draw=red},
]
coordinates{
(  000.0  ,	 0.693147180 )
(  400.0  ,	 0.192264146 )
(  800.0  ,	 0.147592514 )
( 1200.0  ,	 0.147592514 )
( 1600.0  ,	 0.151489140 )
( 2000.0  ,	 0.148624409 )
( 2400.0  ,	 0.134453886 )
( 2800.0  ,	 0.133401812 )
( 3200.0  ,	 0.129717155 )
( 3600.0  ,	 0.127928063 )
};\label{plots:dapd}

        \nextgroupplot[
title={\texttt{rcv1}},
title style={at={(axis description cs:0.5,0.93)},anchor=south},
xmin=0,
xmax=3600,
xmajorgrids,
ymajorgrids]

\addplot [
color=magenta,
line width=1.0pt,
mark size=2.0pt,
mark=square,
mark options={solid,,fill=white,draw=magenta},
]
coordinates{
(  000.0  ,	 0.693147180)
(  400.0  ,	 0.2226550814)
(  800.0  ,	 0.2034553810)
( 1200.0  ,	 0.1924458241)
( 1600.0  ,	 0.1847568984)
( 2000.0  ,	 0.1788868448)
( 2400.0  ,	 0.1741666138)
( 2800.0  ,	 0.1702374163)
( 3200.0  ,	 0.1668841412)
( 3600.0  ,	 0.1639677717)
};\label{plots:dgd}

\addplot [
color=blue,
line width=1.0pt,
mark size=1.8pt,
mark=triangle*,
mark options={solid,,fill=white,draw=blue},
]
coordinates{
(  000.0  , 0.6931471806)
(  400.0  , 0.5101871034)
(  800.0  , 0.3427174885)
( 1200.0  , 0.2695722672)
( 1600.0  , 0.2810675788)
( 2000.0  , 0.2255061845)
( 2400.0  , 0.2292482741)
( 2800.0  , 0.2301961952)
( 3200.0  , 0.2163317944)
( 3600.0  , 0.2067466149)
};\label{plots:pwg}

\addplot [
color=black,
line width=1.0pt,
mark size=1.4pt,
mark=square*,
mark options={solid,,fill=white,draw=black},
]
coordinates{
(  000.0  , 0.6931471806)
(  400.0  , 0.3253581521)
(  800.0  , 0.2927333714)
( 1200.0  , 0.2635785775)
( 1600.0  , 0.2608966047)
( 2000.0  , 0.2452571961)
( 2400.0  , 0.2319612090)
( 2800.0  , 0.2120348835)
( 3200.0  , 0.2002494831)
( 3600.0  , 0.2002494831)
};\label{plots:abg}

\addplot [
color=red,
line width=1.0pt,
mark size=1.5pt,
mark=triangle*,
mark options={solid,fill=white,draw=red},
]
coordinates{
(  000.0  , 0.6931471806)
(  400.0  , 0.5855093832)
(  800.0  , 0.2792802835)
( 1200.0  , 0.1952896382)
( 1600.0  , 0.1487922720)
( 2000.0  , 0.1314744795)
( 2400.0  , 0.1048575734)
( 2800.0  , 0.0946782417)
( 3200.0  , 0.0916883266)
( 3600.0  , 0.082817352)
};\label{plots:dapd}

\coordinate (bot) at (rel axis cs:1,0);
\coordinate (bot2) at (rel axis cs:1,-0.15);

\end{groupplot}
\path (top-|current bounding box.west)-- 
          node[anchor=south,rotate=90] {functional cost} 
          (bot-|current bounding box.west);
\path (bot2-|current bounding box.west)-- 
          node[anchor=north] {number of local gradients computed} 
          (bot2-|current bounding box.east);
\path (top|-current bounding box.north)--
      coordinate(legendpos)
      (bot|-current bounding box.north);
\matrix[
    matrix of nodes,
    anchor=south,
    draw,
    inner sep=0.2em,
    draw
  ]at([yshift=1ex]legendpos)
  {
    \ref{plots:dgd}& {\footnotesize DGD} &[5pt]
    \ref{plots:abg}& {\footnotesize ABG} \\
    \ref{plots:pwg}& {\footnotesize PWG} &[5pt]
    \ref{plots:dapd}&{\footnotesize  DAPD} \\};
\end{tikzpicture}
  \caption{Comparison of distributed algorithms on a $5\times 5$ grid.}
\label{fig:distopt}
\end{figure}

\begin{figure}[h!]
  \centering
  \begin{tikzpicture}
    \begin{groupplot}[group style={group size= 2 by 1},height=0.6\figureheight,width=0.75\figurewidth]

        \nextgroupplot[
title={\texttt{ \scriptsize rcv1, $10\times 10$ grid}},
title style={at={(axis description cs:0.5,0.93)},anchor=south},
xmin=0,
xmax=3600,
xmajorgrids,
ymajorgrids,
]

\addplot [
color=magenta,
line width=1.0pt,
mark size=2.0pt,
mark=square,
mark options={solid,,fill=white,draw=magenta},
]
coordinates{
(0 	,0.6931471806)
(400 	,0.411614928)
(800 	,0.381428499)
(1200 	,0.362237697)
(1600 	,0.348341063)
(2000 	,0.337612563)
(2400 	,0.328965838)
(2800 	,0.321773777)
(3200 	,0.315646603)
(3600 	,0.310328316)
};\label{plots:dgd}

\addplot [
color=blue,
line width=1.0pt,
mark size=1.8pt,
mark=triangle*,
mark options={solid,,fill=white,draw=blue},
]
coordinates{
(000  , 0.6931471806 )
(400 ,0.4045125296 )
(800 ,0.3034318120 )
(1200,0.2855616037 )
(1600,0.2287505935 )
(2000,0.2340736002 )
(2400,0.1762020474 )
(2800,0.1853240999 )
(3200,0.1792963629 )
(3600,0.1826410775 )
};\label{plots:pwg}

\addplot [
color=black,
line width=1.0pt,
mark size=1.4pt,
mark=square*,
mark options={solid,,fill=white,draw=black},
]
coordinates{
(000 , 0.6931471806 )
(400 , 0.6697096702 )
(800 , 0.2942965351 )
(1200, 0.2471375860 )
(1600, 0.2348206963 )
(2000, 0.2303403102 )
(2400, 0.2508702781 )
(2800, 0.2110502763 )
(3200, 0.2026678806 )
(3600, 0.1965281742 )
};\label{plots:abg}

\addplot [
color=red,
line width=1.0pt,
mark size=1.5pt,
mark=triangle*,
mark options={solid,fill=white,draw=red},
]
coordinates{
(0 	 , 	 0.6931471806)
(400 	 , 	 0.4013624820)
(800 	 , 	 0.2895504063)
(1200 	 , 	 0.3190781384)
(1600 	 , 	 0.2952178457)
(2000 	 , 	 0.2248980272)
(2400 	 , 	 0.2358225621)
(2800 	 , 	 0.2246819814)
(3200 	 , 	 0.1887560508)
(3600 	 , 	 0.1798911940)
};\label{plots:dapd}

\coordinate (top) at (rel axis cs:0,1);

        \nextgroupplot[
title={\texttt{\scriptsize rcv1, $50$-nodes Complete graph}},
title style={at={(axis description cs:0.5,0.93)},anchor=south},
xmin=0,
xmax=3000,
xmajorgrids,
ymajorgrids]

\addplot [
color=magenta,
line width=1.0pt,
mark size=2.0pt,
mark=square,
mark options={solid,,fill=white,draw=magenta},
]
coordinates{
(  000.0  ,	 0.6931471803)
(  400.0  ,	 0.2924269517)
(  800.0  ,	 0.2695282062)
( 1200.0  ,	 0.2559161785)
( 1600.0  ,	 0.2463605423)
( 2000.0  ,	 0.2390618054)
( 2400.0  ,	 0.2331919456)
( 2800.0  ,	 0.2283036297)
( 3200.0  ,	 0.2241287725)
( 3600.0  ,	 0.2204945231)
};\label{plots:dgd}

\addplot [
color=blue,
line width=1.0pt,
mark size=1.8pt,
mark=triangle*,
mark options={solid,,fill=white,draw=blue},
]
coordinates{
(  000.0  ,	 0.6931471806)
(  400.0  ,	 0.333581521)
(  800.0  ,	 0.30333714)
( 1200.0  ,	 0.285785775)
( 1600.0  ,	 0.278966047)
( 2000.0  ,	 0.2552571961)
( 2400.0  ,	 0.240612090)
( 2800.0  ,	 0.220348835)
( 3200.0  ,	 0.202494831)
( 3600.0  ,	 0.2002494831)
};\label{plots:pwg}

\addplot [
color=black,
line width=1.0pt,
mark size=1.4pt,
mark=square*,
mark options={solid,,fill=white,draw=black},
]
coordinates{
(  000.0  , 0.6931471806)
(  400.0  , 0.3253581521)
(  800.0  , 0.27333714)
( 1200.0  , 0.2635785775)
( 1600.0  , 0.2408966047)
( 2000.0  , 0.2352571961)
( 2400.0  , 0.2319612090)
( 2800.0  , 0.2120348835)
( 3200.0  , 0.202494831)
( 3600.0  , 0.2002494831)
};\label{plots:abg}

\addplot [
color=red,
line width=1.0pt,
mark size=1.5pt,
mark=triangle*,
mark options={solid,fill=white,draw=red},
]
coordinates{
(  000.0  , 0.6931471806)
(  400.0  , 0.455093832)
(  800.0  , 0.282802835)
( 1200.0  , 0.1952896382)
( 1600.0  , 0.1487922720)
( 2000.0  , 0.1314744795)
( 2400.0  , 0.1048575734)
( 2800.0  , 0.0946782417)
( 3200.0  , 0.0916883266)
( 3600.0  , 0.0712817352)
};\label{plots:dapd}

\coordinate (bot) at (rel axis cs:1,0);
\coordinate (bot2) at (rel axis cs:1,-0.15);

\end{groupplot}
\path (top-|current bounding box.west)-- 
          node[anchor=south,rotate=90] {functional cost} 
          (bot-|current bounding box.west);
\path (bot2-|current bounding box.west)-- 
          node[anchor=north] {number of local gradients computed} 
          (bot2-|current bounding box.east);
\path (top|-current bounding box.north)--
      coordinate(legendpos)
      (bot|-current bounding box.north);
\matrix[
    matrix of nodes,
    anchor=south,
    draw,
    inner sep=0.2em,
    draw
  ]at([yshift=1ex]legendpos)
  {
    \ref{plots:dgd}& {\footnotesize DGD} &[5pt]
    \ref{plots:abg}& {\footnotesize ABG} \\
    \ref{plots:pwg}& {\footnotesize PWG} &[5pt]
    \ref{plots:dapd}&{\footnotesize  DAPD} \\};
\end{tikzpicture}
  \caption{Comparison between different networks on \texttt{rcv1} dataset.}
\label{fig:distopt2}
\end{figure}
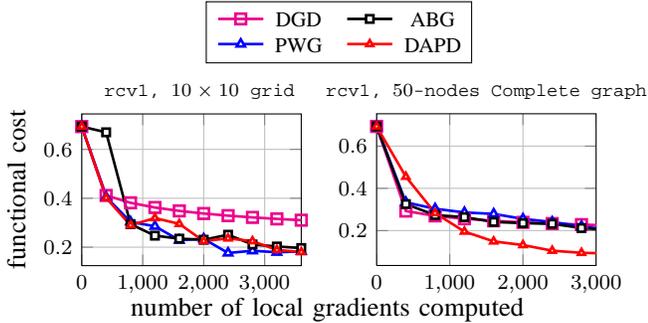

\section{Conclusions and Perspectives}

This paper introduced a general framework for stochastic coordinate descent.
The framework was used on a new algorithm called ADMM+ which has roots in 
a recent work by V\~u and Condat. As a byproduct, we obtained 
an asynchronous distributed algorithm
which enables the processing of distinct blocks on different machines.
Future works include an analysis of the convergence rate of our algorithms along with efficient stepsizes strategies.

\appendix 

\subsection{Proof of Theorem~\ref{the:cv}}
\label{sec:proofCV}

By setting ${\cal E} = {\cal S}$ and by assuming that $\cal E$ is equipped 
with the same inner product as $\cal Y$, one can notice that the functions 
$\bar f = f\circ M^{-1}$, $\bar g=g\circ M^{-1}$ and $h$ satisfy the conditions
of Theorem~\ref{vu}. Moreover, since $(\bar f + \bar g)^* = 
(f +g)^* \circ M^*$, one can also notice that $(x^\star, \lambda^\star)$ is 
a primal-dual point associated with Eq.~\eqref{duality} if and only if 
$(M x^\star, \lambda^\star)$ is a primal-dual point associated with 
Eq.~\eqref{pd-vu}.  \\
To recover ADMM+ from the iterations~\eqref{vu-l}--\eqref{vu-y}, the 
starting point is Moreau's identity~\cite[Th.~14.3]{livre-combettes} which
reads 
\[
\prox_{\rho^{-1}h^*}(x)+\rho^{-1}\prox_{\rho h}(\rho x) = x\,.
\] 
Setting $x^k = M^{-1} y^k$ and 
\begin{align*}
z^{k+1} &= \prox_{\rho h}( y^k + \rho \lambda^k ) \\
&= \argmin_{w\in{\cal Y}} \Bigl[ 
h(w) + \frac{\| w - (M x^k + \rho \lambda^k) \|^2}{2\rho} \Bigr]  \ , 
\end{align*} 
Equation~\eqref{vu-l} can be rewritten thanks to Moreau's identity 
\[
\lambda^{k+1} = \lambda^k + \rho^{-1}( M x^k - z^{k+1} )  \ .
\]
Now, Equation~\eqref{vu-y} can be rewritten as 
$$
y^{k+1} \!\!= \!\!\argmin_{w\in{\cal S}} \bar g(w)+\langle \nabla \bar f(y^k),\!w\rangle +\frac {\|w\!-\!y^k\!+\!\tau(2\lambda^{k+1}\!\!-\!\lambda^k)\|^2}{2\tau}
$$
Upon noting that $\bar g(Mx)=g(x)$ and $\langle\nabla \bar f(y^k),Mx\rangle =
\langle (M^{-1})^* \nabla f(M^{-1} M x^{k}),Mx\rangle = 
\langle \nabla f(x^{k}),x\rangle$, the above equation becomes
$$
x^{k+1} \!\!= \!\!\argmin_{w\in{\cal X}} g(w)
+\langle \nabla f(x^{k}),\!w\rangle +\frac {\|Mw\!-\!u^{k+1}\!\!+\!\tau\lambda^{k+1}\|^2}{2\tau}
$$
where 
\begin{align*} 
u^{k+1} &= Mx^{k}+\tau(\lambda^{k}-\lambda^{k+1}) \\
 &= (1 - \rho^{-1}\tau) M x^{k} + \tau\rho^{-1} z^{k+1} . 
\end{align*} 
The iterates $(z^{k+1}, \lambda^{k+1}, u^{k+1}, x^{k+1} )$ are those of the 
ADMM+. 

\subsection{Proof of Theorem~\ref{the:randomKM}}
\label{sec:proofrandomKM}
The main idea behind the proof can be found in~\cite{iut-cdc13}. 
Define the operator $\oU = (1-\eta_k)\oI+\eta_k\oT$ (we omit the index $k$ in $\oU$ to simplify the notation); similarly, define  $\oU^{(\kappa)} = (1-\eta_k)\oI+\eta_k\hat\oT^{(\kappa)}$.
Remark that the operator $\oU$ is $(\alpha\eta_k)$-averaged. 

The iteration~(\ref{eq:3}) reads $x^{k+1}=\oU^{(\xi^{k+1})} x^k$.
Set $p_\kappa = {\mathbb P}(\xi_1=\kappa)$ for any $\kappa\in 2^{\cal J}$.
Denote by $\|x\|^2=\langle x,x\rangle$ the squared norm in $\cH$. Define a new inner product
$x \bullet y = \sum_{j=1}^J q_j \langle x_j, y_j \rangle_{j}$ on $\cH$ where
$
q_j^{-1} = \sum_{\kappa\in 2^{\cal J}} p_\kappa \un_{\{j\in \kappa\}}
$
and let $\leftnorm x \rightnorm^2 = x \bullet x$ 
be its associated squared norm. Consider any $x^\star\in \fix(\oT)$.
Conditionally to the sigma-field $\mathcal{F}^k = \sigma(\xi_1,\ldots, \xi^k)$ we have
\begin{align*}
 & \EE[\leftnorm x^{k+1} - x^\star\rightnorm^2 \, | \, 
 \mathcal{F}^k ] 
= \sum_{\kappa\in 2^{\cal J}} p_\kappa \leftnorm \hat{\oU}^{(\kappa)}x^k -  x^\star\rightnorm^2 \\
&= \sum_{\kappa\in 2^{\cal J}} p_\kappa\sum_{j\in\kappa}q_j\|\oU_j x^k-x^\star_j\|^2
+ \sum_{\kappa\in 2^{\cal J}} p_\kappa\sum_{j\notin\kappa}q_j\|x^k_j-x^\star_j\|^2\\
&= \leftnorm x^k-x^\star\rightnorm^2\!+\!\!\sum_{\kappa\in 2^{\cal J}} \!p_\kappa\!\sum_{j\in\kappa}q_j\!\left(\|\oU_j x^k-x^\star_j\|^2\!-\!\|x^k_j-x^\star_j\|^2\right) \\
&= \leftnorm x^k-x^\star\rightnorm^2+\sum_{j=1}^J\left(\|\oU_j x^k-x^\star_j\|^2-\|x^k_j-x^\star_j\|^2\right) \\
&=\leftnorm x^k-x^\star\rightnorm^2 +\left(\|\oU x^k-x^\star\|^2-\|x^k-x^\star\|^2\right)\,.
\end{align*}
Using that $\oU$ is $(\alpha\eta_k)$-averaged and that $x^\star$ is a fixed point of $\oU$, the term enclosed in the parentheses is no larger than $-\frac{1-\alpha \eta_k}{\alpha\eta_k}\|(\oI-\oU) x^k\|^2$\,. 
As $\oI-\oU = \eta_k(\oI-\oT)$, we obtain:
\begin{multline}
  \EE[\leftnorm x^{k+1} - x^\star\rightnorm^2 \, | \, \mathcal{F}^k ]
  \leq \leftnorm x^k-x^\star\rightnorm^2 \\ -
  \eta_k(1-\alpha\eta_k)\|(\oI-\oT) x^k \|^2
\label{eq:doob}
\end{multline}
which shows that $\leftnorm x^{k} - x^\star\rightnorm^2$ is a 
nonnegative supermartingale with respect to the filtration $({\mathcal F}_k)$. 
As such, it converges with probability one towards a random variable 
that is finite almost everywhere.

Given a countable dense subset $H$ of $\fix(\oT)$, there is a probability
one set on which $\leftnorm x^k - {\boldsymbol x} 
\rightnorm \to X_{\boldsymbol x} \in [0, \infty)$ for all 
${\boldsymbol x} \in H$. 
Let $x^\star \in \fix(\oT)$, let $\varepsilon > 0$, and choose 
${\boldsymbol x} \in H$ such that $\leftnorm x^\star - \boldsymbol x
\rightnorm \leq \varepsilon$. With probability one, we have  
\[
\leftnorm x^k -x^\star \rightnorm \leq  
\leftnorm x^k - \boldsymbol x \rightnorm + 
\leftnorm \boldsymbol x - x^\star \vphantom{x^k} \rightnorm \leq 
X_{\boldsymbol x} + 2 \varepsilon 
\]  
for $k$ large enough. Similarly, 
$\leftnorm x^k - x^\star \rightnorm \geq  
X_{\boldsymbol x} - 2 \varepsilon$ for $k$ large enough.  
We therefore obtain: 
\begin{description}
\item[{\bf C1 :}] There is a probability one set on which
$\leftnorm x^k - x^\star \rightnorm$ converges for every 
$x^\star \in \fix(\oT)$. 
\end{description} 
Getting back to \eqref{eq:doob}, taking the expectations on
both sides of this inequality and iterating over $k$, we obtain 
\[
\sum_{k=0}^\infty \eta_k(1-\alpha\eta_k)\EE \|(\oI-\oT) x^k\|^2 \leq 
\leftnorm x^0 - x^\star\rightnorm^2 . 
\]

Using the assumption on $(\eta_k)_{k\in \NN}$, it is straightforward to see 
that $\sum_{k=0}^\infty \eta_k(1-\alpha\eta_k) = + \infty$ and thus that 
$\sum_{k=0}^\infty \EE \|(\oI-\oT) x^k\|^2$ is finite. By Markov's inequality 
and Borel Cantelli's lemma, we therefore obtain: 
\begin{description}
\item[{\bf C2 :}] $(\oI-\oT) x^k \to 0$ almost surely. 
\end{description} 
We now consider an elementary event in the probability one set where {\bf C1} 
and {\bf C2} hold. On this event, 
since $\leftnorm x^k - x^\star \rightnorm$ converges for 
$x^\star \in \fix(\oT)$, the sequence $(x^k)_{k\in\NN}$ is bounded. 
Since $\oT$ is $\alpha$-averaged, it is continuous, and {\bf C2} shows that 
all the accumulation points of $(x^k)_{k\in\NN}$ are in $\fix(\oT)$. It remains to show
that these accumulation points reduce to one point. Assume that 
$x_1^\star$ is an accumulation point. By {\bf C1}, 
$\leftnorm x^k - x^\star_1 \rightnorm$ converges. Therefore, 
$\lim \leftnorm x^k - x^\star_1 \rightnorm = 
\liminf \leftnorm x^k - x^\star_1 \rightnorm = 0$, which shows that
$x^\star_1$ is unique. 

\subsection{Proof of Theorem~\ref{the:DADMM+}}
\label{app:proof-DADMM+}

The proof simply consists in checking that the assumptions of Theorem~\ref{the:cv} are satisfied.
To that end, we compute the Lipschitz constant $L$ of $\nabla(f\circ M^{-1})$ as a function of $\bar L$.
Recall that $\cal S$ is the image of $M$. For any $y\in {\cal S}$, note that
\begin{equation}
\label{eq:sacrecoeur}
\nabla (f\circ M^{-1})(y) = M(M^*M)^{-1} \nabla f(M^{-1}y)\,.
\end{equation}
Using the definition of $M$, the operator $M^*M$ is diagonal. More precisely, for any $x\in {\mathbb R}^{N}$, say $x=(x_n)_{n\in V}$,
the $n$th component of $(M^*M)x$ coincides with $d_nx_n$ where $d_n=\mathrm{card}\{m\in V\,:\,n\sim m\}$ is the degree of node $n$ in the graph~$G$.
Thus, $\|M(M^*M)^{-1}x\|^2 = \sum_{n\in V} d_n^{-1}\|x_n\|^2$. As a consequence of the latter equality and~(\ref{eq:sacrecoeur}), 
for any $(y,y')\in{\cal S}^2$, say $y=Mx$ and $y'=Mx'$, one has
\begin{multline*}
  \|\nabla (f\circ M^{-1})(y)-\nabla (f\circ M^{-1})(y')\|^2 \\ =
    \sum_n d_n^{-1}\|\nabla f_n(x_n)-\nabla f_n(x_n')\|^2\,.
  \end{multline*}
Under the stated hypotheses, we can write for all $n$,  $\|\nabla f_n(x_n)-\nabla f_n(x_n')\|^2\leq \bar L^2\|x_n-x_n'\|^2$.
Thus, 
\begin{equation}
\|\nabla (f\circ M^{-1})(y)-\nabla (f\circ M^{-1})(y')\|^2 \leq (\bar L^2/d_{\mathrm{min}}) \|x-x'\|^2\label{eq:pantheon}
\end{equation}
where $d_{\mathrm{min}}=\min(d_n:n\in V)$. On the otherhand, $\|y-y'\|^2=\|M(x-x')\|^2=\sum_n d_n\|x_n-x_n'\|^2\geq d_{\mathrm{min}} \|x-x'\|^2$.
Plugging the latter inequality into~(\ref{eq:pantheon}), we finally obtain 
$\|\nabla (f\circ M^{-1})(y)-\nabla (f\circ M^{-1})(y')\|^2\leq (\bar L/d_{\mathrm{min}})^2 \|x-x'\|^2$. This proves that 
$\nabla(f\circ M)$ is Lipschitz continuous with constant $L=\bar L/d_{\mathrm{min}}$.
The final result follows by immediate application of Theorem~\ref{the:cv}.
\color{black}

\subsection{Proof of Theorem~\ref{th:dapd}}
\label{prf:dapd}
Let $(\bar f, \bar g, h) = ( f\circ M^{-1}, g \circ M^{-1}, h)$ where 
$f,g,h$ and $M$ are those of Problem~\eqref{eq:pb-od}. For these functions, 
write Equations~\eqref{alg-vu} as 
$(\lambda^{k+1}, y^{k+1} ) = \oT (\lambda^k, y^k)$. By Lemma~\ref{lem:cv-fbpd},
the operator $\oT$ is an $\alpha$-averaged operator acting on the space 
${\cal H} = {\cal Y} \times {\cal S}$, where ${\cal S}$ is the image of ${\cal X}^{N}$ 
by $M$. For any $n \in V$, let $S_n$ be the selection 
operator on $\cal H$ defined as 
$S_n(\lambda, Mx) = 
( ( \lambda_{\epsilon}(n) )_{\epsilon \in E \, : \, n \in \epsilon}, x_n )$. 
Then it is easy to see that up to an element reordering, 
${\cal H} = S_1({\cal H}) \times \cdots \times S_{N}({\cal H})$. 
Identifying the set $\cal J$ introduced before the statement of 
Theorem~\ref{the:randomKM} with $V$, the operator $\oT^{(\xi^k)}$ is 
defined as follows: if $n \in \xi^k$, then 
$S_n(\oT^{(\xi^k)}(\lambda, Mx)) = S_n(\oT(\lambda, Mx))$ while if 
$n \not\in \xi^k$, then 
$S_n(\oT^{(\xi^k)}(\lambda, Mx)) = S_n(\lambda, Mx)$. 
We know by Theorem~\ref{the:randomKM} that the sequence 
$(\lambda^{k+1}, M x^{k+1}) = \oT^{(\xi^{k+1})}(\lambda^{k}, M x^{k})$ 
converges almost surely to a primal-dual point of Problem~\eqref{pd-vu}. This 
implies by Lemma~\ref{consensus} that the sequence $x^k$ converges almost 
surely to $(x^\star, \ldots, x^\star)$ where $x^\star$ is a minimizer of 
Problem~\eqref{eq:pb-distrib}. \\
We therefore need to prove that the operator $\oT^{(\xi^{k+1})}$ is translated
into the DAPD algorithm. The definition~\eqref{def:h} of $h$ shows that 
\[
h^*(\phi) = \sum_{\epsilon \in E} \iota_{{\cal C}_2^\perp}(\phi_\epsilon) 
\]
where ${\cal C}_2^\perp = \{ (x,-x)\, : \, x \in {\cal X} \}$. Therefore,
writing 
\[
(\eta^{k+1}, q^{k+1} = Mv^{k+1}) = \oT(\lambda^k, y^k =Mx^k), 
\]
Equation~\eqref{vu-l} shows that 
\[
\eta^{k+1}_\epsilon = 
\text{proj}_{{\cal C}_2^\perp}(\lambda^k_\epsilon + \rho^{-1} y^k_\epsilon) .
\]  
Notice that contrary to the case of the synchronous algorithm DADMM+, 
there is no reason here for which 
$\proj_{{\cal C}_2^\perp}(\lambda^k_\epsilon) = 0$.
Getting back to 
$(\lambda^{k+1}, M x^{k+1}) = \oT^{(\xi^{k+1})}(\lambda^{k}, y^k = M x^{k})$, 
we therefore obtain that for all $n \in \xi^{k+1}$ and all $m \sim n$, 
\begin{align*} 
\lambda_{\{n,m\}}^{k+1}(n) &= 
    \frac{\lambda_{\{n,m\}}^k(n) - \lambda_{\{n,m\}}^k(m)}{2} \\
& 
\ \ \ \ \ \ \ \ 
+ \frac{y_{\{n,m\}}^k(n) - y_{\{n,m\}}^k(m)}{2\rho} \\
&= \frac{\lambda_{\{n,m\}}^k(n) - \lambda_{\{n,m\}}^k(m)}{2} +
    \frac{x^k_n - x^k_m}{2\rho} . 
\end{align*}
Recall now that Eq.~\eqref{vu-y} can be rewritten as 
$$
q^{k+1} \!\!= \!\!\argmin_{w\in{\cal S}} \bar g(w)+\langle \nabla \bar f(y^k),\!w\rangle +\frac {\|w\!-\!y^k\!+\!\tau(2\lambda^{k+1}\!\!-\!\lambda^k)\|^2}{2\tau}
$$
Upon noting that $\bar g(Mx)=g(x)$ and $\langle\nabla \bar f(y^k),Mx\rangle =
\langle (M^{-1})^* \nabla f(M^{-1} M x^{k}),Mx\rangle = 
\langle \nabla f(x^{k}),x\rangle$, the above equation becomes
\begin{multline*} 
v^{k+1} = \argmin_{w\in{\cal X}} g(w)
+\langle \nabla f(x^{k}),w\rangle \\
+\frac {\|M(w-x^{k})+\tau (2\lambda^{k+1}-\lambda^k)\|^2}{2\tau} . 
\end{multline*} 
Recall that $(M^*M x)_n = d_n x_n$. Hence, for all $n \in \xi^{k+1}$, we get 
after some computations  
\begin{multline*} 
x^{k+1}_n = 
\prox_{\tau g_n/d_n} \Bigl[ x_n^k - \frac{\tau}{d_n}\nabla f_n(x_n^k) \\ 
- \frac{\tau}{d_n} ( M^*(2\lambda^{k+1} - \lambda^k))_n \Bigr] . 
\end{multline*} 
Using the identity $(M^*y)_n = \sum_{m : \{n,m\} \in E} y_{\{n,m\}}(n)$, one
can check that this equation coincides with the $x-$update equation in the 
DAPD algorithm. 

\section*{Acknowledgement}

 The authors would like to thank the anonymous reviewers for their useful comments.

\bibliographystyle{IEEEbib}
\bibliography{math}

\end{document}